\documentclass[12pt]{article}

\usepackage{amsmath, amsthm,amssymb}

\newtheorem{theorem}{Theorem}
\newtheorem{lemma}{Lemma}

\DeclareMathOperator{\rank}{rank}

\begin{document}

\title{Classification of affine operators up to biregular conjugacy}

\author{Tetiana Budnitska\\Institute of Mathematics, Tereshchenkivska 3, Kyiv, Ukraine\\ budnitska\_t@ukr.net
    \and
Nadiya Budnitska\\                     Kyiv Taras Shevchenko University\\
Volodymytska 64, Kyiv, Ukraine\\
nadya\_vb@ukr.net}

\date{}

\maketitle
\begin{abstract}
Let $f(x)=Ax+b$ and $g(x)=Cx+d$ be two affine operators given by $n$-by-$n$ matrices $A$ and $C$ and vectors $b$ and $d$ over a field $\mathbb F$. They are said to be  biregularly conjugate if $f=h^{-1}gh$ for some bijection $h: \mathbb F^n\to \mathbb F^n$ being biregular, this means that the coordinate functions of $h$ and $h^{-1}$ are polynomials. Over an algebraically closed field of characteristic $0$,
we obtain necessary and sufficient
conditions of biregular conjugacy of affine operators and give a canonical form of an affine operator up to biregular conjugacy. These results for bijective affine operators were obtained by J.~Blanc [Conjugacy classes of affine automorphisms of $\mathbb K^n$ and linear automorphisms of $\mathbb P^n$ in the Cremona groups, Manuscripta Math. 119 (2006) 225--241].



\end{abstract}

\section{Introduction and theorem}

In this article, affine operators are classified up to biregular conjugacy. All matrices and vector spaces that we consider are over an algebraically closed field $\mathbb F$ of zero characteristic.

A map $f: \mathbb F^n \rightarrow \mathbb F^n$ of the form $f(x)=Ax+b$ with $A \in \mathbb F^{\,n\times n}$ and $b \in \mathbb F^n$ is called an \emph{affine operator} and $A$ is called its \emph{matrix}. 
 Two maps $f$, $g: \mathbb F^n \rightarrow \mathbb F^n$ are \emph{biregularly  conjugate} if $g=h^{-1} f h$ for some bijection $h: \mathbb F^n \rightarrow \mathbb F^n$ being \emph{biregular}, this means that $h$ and $h^{-1}$ have the form
\begin{equation*}
(x_1,\ldots,x_n) \mapsto (h_1(x_1, \ldots, x_n),\ldots,h_n(x_1, \ldots, x_n)),
\end{equation*}
in which all $h_i$ are polynomials over $\mathbb F$.
The group of biregular operators on $\mathbb F^n$ is called the \textit{affine Cremona group}.

Two matrices $A,B\in\mathbb{F}^{n\times n}$ are \emph{similar} if $A=S^{-1}BS$ for some nonsingular $S\in\mathbb F^{n\times n}$.
An element $x$ is called a \emph{fixed point} of $f$ if $f(x)=x$. An affine operator without fixed point can not be biregularly conjugate to an affine operator with fixed point since if two maps on $\mathbb F^n$ are biregularly conjugate then they have the same number of fixed points.

J\'er\'emy Blanc classified bijective affine operators up to biregular conjugacy as follows.

\begin{theorem}[Blanc \cite{Blanc}]\label{Blanc}
Let $\mathbb F$ be an algebraically closed field of characteristic~$0$.

\begin{itemize}
  \item[\rm(a)] Two bijective affine operators $\mathbb F^n\to \mathbb F^n$ with fixed points are biregularly conjugate if and only if their matrices are similar.

  \item[\rm(b)] Each bijective affine operator $f:\mathbb F^n\to \mathbb F^n$ without fixed point is biregularly conjugate to an ``almost-diagonal'' affine operator
      \begin{equation}\label{diag-autom}
(x_1, x_2, \ldots, x_n)\mapsto (x_1+1, \alpha_2 x_2, \dots, \alpha_n x_n),
\end{equation}
in which $1,\alpha_2, \dots, \alpha_n\in \mathbb F\setminus 0$ are all eigenvalues of the matrix of $f$ repeated according to their multiplicities. The affine operator \eqref{diag-autom} is uniquely determined by $f$, up to permutation of $\alpha_2,\dots,\alpha_n$.

\end{itemize}
\end{theorem}

Each square matrix
\begin{equation}\label{+-0a}
  A\quad\text{is similar to}\quad A_* \oplus A_{\circ},
\end{equation}
in which $A_*$ is nonsingular and $A_{\circ}$ is nilpotent.

In this article, we proof the following theorem, in which
Theorem~\ref{Blanc} is extended to arbitrary affine operators.

\begin{theorem}\label{all-aff-bir}
Let $\mathbb F$ be an algebraically closed field of characteristic~$0$.
\begin{itemize}
\item[{\rm (a)}] The criterion of biregular conjugacy of two affine operators $f(x)=Ax+b$ and $g(x)= Cx+d$ over $\mathbb F$ is the following:
\begin{itemize}
\item[\rm(i)] If $f$ and $g$ have fixed points, then $f$ and $g$ are biregularly conjugate if and only if their matrices $A$ and $C$ are similar.

\item[\rm(ii)] If $f$ and $g$ have no fixed points, then $f$ and $g$ are biregularly conjugate if and only if the nonsingular matrices $A_*$ and $C_*$ $($defined in \eqref{+-0a}$)$ have the same eigenvalues with the same multiplicities, and the nilpotent matrices $A_{\circ}$ and $C_{\circ}$ are similar.
\end{itemize}

\item[{\rm (b)}] A canonical form of an affine operator $f(x)=Ax+b$ under biregular conjugacy is the following:
\begin{itemize}
\item[\rm(i)] If $f$ has a fixed point, then $f$ is biregularly conjugate to the linear operator
$x \mapsto Jx,$
in which $J$ is the Jordan canonical form of $A$ uniquely determined by $f$ up to permutation of blocks.

\item[\rm(ii)] If $f$ has no fixed point, then $f$ is biregularly conjugate to
\begin{equation}\label{canon-ne-bijec}
\begin{bmatrix}
  x_1\\ x_2\\ \vdots\\ x_n
\end{bmatrix}\mapsto \begin{bmatrix} x_1+1\\ \alpha_2 x_2\\ \vdots\\ \alpha_k x_k\\J_{\circ}\tilde{x}
\end{bmatrix},\qquad \tilde{x}:= \begin{bmatrix}
  x_{k+1}\\ x_{k+2}\\ \vdots\\ x_n
\end{bmatrix},
\end{equation}
in which $1,\alpha_2,\dots,\alpha _k\in \mathbb F\setminus 0$ are all the eigenvalues of $A_*$ repeated according to their multiplicities, and $J_{\circ}$ is the Jordan canonical form of $A_{\circ}$. The affine operator \eqref{canon-ne-bijec} is uniquely determined by $f$, up to permutation of eigenvalues $\alpha_2,\dots,\alpha _k$ and permutation of blocks in $J_{\circ}$.
\end{itemize}
\end{itemize}
\end{theorem}

\section{Proof of Theorem~\ref{all-aff-bir}}

\begin{lemma}[\cite{Blanc,st-1}]\label{pro-neryx-tochk}
An affine operator $f(x)=Ax+b$ over $\mathbb F$  has a fixed point if and only if  $f$ is biregularly conjugate to its linear part $f_{\text{\rm lin}}(x):=Ax$.
\end{lemma}

\begin{proof}
If $p$ is a fixed point of $f$, then $f$ and $f_{\text{\rm lin}}$ are biregularly conjugate via  $h(x):=x+p$ since for each $x$
\begin{align*}
(h^{-1}fh)(x)&= (h^{-1}f)(x+p)= h^{-1}(A(x+p)+b)\\&= h^{-1}(Ax+(p-b)+b)=h^{-1}(Ax+p)=Ax=f_{\text{\rm lin}}(x)
\end{align*}
and so $f_{\text{\rm lin}}=h^{-1}fh$.

Conversely, let $f$ and $ f_{\text{\rm lin}}$ be biregularly conjugate. Since $f_{\text{\rm lin}}(0)=0$, $f_{\text{\rm lin}}$ has a fixed point. Because biregularly conjugate maps have the same number of fixed points, $f$ has a fixed point too.
\end{proof}

It is convenient to give an affine operator $f(x)=Ax+b$ by the pair $(A,b)$ and identify $f$ with this pair.

For two affine operators $f$ on $\mathbb F^m$ and $g$ on $\mathbb F^n,$ their \emph{direct sum} is the affine operator $f\oplus g$ on $\mathbb F^{m+n}$ defined as follows:
\[
(f\oplus g)(\begin{bmatrix}x \\y\end{bmatrix}) :=\begin{bmatrix}
                     f(x) \\
                     g(y) \\
                   \end{bmatrix}.
\]
Thus,
\begin{equation*}\label{rwi}
(A,b)\oplus (C,d)=\left(\begin{bmatrix}
                     A & 0 \\
                     0 & C \\
                   \end{bmatrix},
\begin{bmatrix}
                     b \\
                     d \\
                   \end{bmatrix}
 \right).
\end{equation*}
We write
$f\sim g$ if $f$ and $g$ are biregularly  conjugate. Clearly,
\begin{equation}\label{xol}
f\sim f'\text{ and } g\sim g' \quad \Longrightarrow \quad f\oplus g \sim f'\oplus g'.
\end{equation}

The statement (i) of Theorem \ref{all-aff-bir}(a) can be proved as ~\cite[Proposition 2]{Blanc}:
by Lemma~\ref{pro-neryx-tochk}, affine operators $f(x)=Ax+b$ and $g(x)= Cx+d$ with fixed points are biregularly conjugate if and only if their linear parts $f_{\text{\rm lin}}(x)=Ax$ and $g_{\text{\rm lin}}(x)=Cx$ are biregularly conjugate if and only if there exists a biregular map $\varphi:\mathbb F^n \to \mathbb F^n$ such that $\varphi f_{\text{\rm lin}}=g_{\text{\rm lin}} \varphi$. Differentiating the last equality and evaluating at zero
gives $D\varphi(0)\,A=C\, D\varphi(0)$, thus $A$ and $C$ are similar. Conversely, if $A=S^{-1}CS$ then $\psi f_{\text{\rm lin}}=g_{\text{\rm lin}} \psi$ with $\psi(x):=Sx$.

\smallskip

The statement (i) of Theorem \ref{all-aff-bir}(b) follows from (i) of Theorem \ref{all-aff-bir}(a).

\smallskip

In the remaining part of the paper, we prove the statement (ii) of Theorem \ref{all-aff-bir}(b), which ensures (ii) of Theorem \ref{all-aff-bir}(a).

\subsection{Reduction to the canonical form}\label{can-form}

Let $f(x)=Ax+b$ be an affine operator over $\mathbb F$ without fixed point. Let $A_*$ and  $A_{\circ}$ be defined by \eqref{+-0a}.
The map $f$ can be reduced to~\eqref{canon-ne-bijec} by transformations of biregular conjugacy in 3 steps:
\begin{itemize}
  \item $f$ is reduced to  the form
\begin{equation}\label{twj}
(A_*,c) \oplus (J_{\circ},s),
\end{equation}
in which $J_{\circ}$ is the Jordan canonical form of~$A_{\circ}$.
We make this reduction by a transformation of \emph{linear conjugacy} $f\mapsto h^{-1}fh$, given by a linear map $h(x)=Sx$ with a nonsingular matrix $S$. Then
$(A,b)\mapsto (S^{-1}AS,S^{-1}b)$; we take $S$ such that $S^{-1}AS=A_* \oplus J_{\circ}$.

  \item \eqref{twj} is reduced  to  the form
\begin{equation}\label{twk}
(A_*,c) \oplus (J_{\circ},0).
\end{equation}
We use \eqref{xol} and $(J_{\circ},s)\sim (J_{\circ},0)$,
 the latter holds by Lemma~\ref{pro-neryx-tochk}. Note that $(A_*,c)$ has no fixed point since if $(x_1,\dots,x_k)$ is a fixed point of $(A_*,c)$ then $(x_1,\dots,x_k,0,\dots,0)$ is a fixed point of \eqref{twk}, but $f$ has no fixed point.

  \item \eqref{twk} is reduced to  the form
\begin{equation}\label{f-can}
(D_{\alpha},e_1) \oplus (J_{\circ},0),\\
\end{equation}
in which
$$
D_{\alpha}:=
\begin{bmatrix}
1 &          &        &   0    \\
  & \alpha_2 &        &        \\
  &          & \ddots &        \\
0 &          &        &\alpha_k
\end{bmatrix},\quad e_1:=\begin{bmatrix}
1\\0\\\vdots\\0
\end{bmatrix},
$$
and $1,\alpha_2,\dots, \alpha _k$ are all the eigenvalues of $A_*$ which are repeated according to their multiplicities  $($note that \eqref{f-can} is another form of \eqref{canon-ne-bijec}$)$.
We use \eqref{xol} and $(A_*,c) \sim (D_{\alpha},e_1)$, the latter follows from Theorem~\ref{Blanc}(b).

\end{itemize}

\subsection{Uniqueness of the canonical form}\label{uniq}

Let
\[
f=f_*\oplus f_{\circ}
,\quad f_*=(D_{\alpha},e_1):{\mathbb F}^p\to{\mathbb F}^p,\quad f_{\circ}=(J_{\circ},0):{\mathbb F}^{n-p}\to{\mathbb F}^{n-p},\]
and
\[
g=g_*\oplus g_{\circ}
,\quad g_*=(D_{\beta },e_1):{\mathbb F}^q\to{\mathbb F}^q,\quad g_{\circ}=(J'_{\circ},0):{\mathbb F}^{n-q}\to{\mathbb F}^{n-q},\]
be two affine operators of the form \eqref{f-can} on ${\mathbb F}^n$, in which $f_*$ and $g_*$ have no fixed points. Let $f$ and $g$ be biregularly conjugate.
For each $i=1,2,\dots$, the images of $f^i$ and $g^i$ are the sets
\begin{equation*}\label{teo}
V_i:= f^i{\mathbb F}^n={\mathbb F}^p\oplus J_{\circ}^i{\mathbb F}^{n-p},\qquad W_i:= g^i{\mathbb F}^n={\mathbb F}^q\oplus J_{\circ}^{\prime i}{\mathbb F}^{n-q},
\end{equation*}
and so they are vector subspaces of ${\mathbb F}^n$ of dimensions
\begin{equation}\label{suk}
\dim V_i=p+\rank J_{\circ}^i,\qquad \dim W_i=q+\rank J_{\circ}'^i.
\end{equation}

Since $f$ and $g$ are biregularly conjugate, there exists a biregular map $h:\mathbb F^n \to \mathbb F^n$ such that $f=h^{-1} g h$.
Then
\begin{equation}\label{hyr}
hf^i= g^i h,\quad
hf^i\mathbb F^n= g^i
h\,\mathbb F^n=g^i
\mathbb F^n,\quad
h\,V_i=W_i.
\end{equation}
By \cite[Chapter 1, Corollary 3.7]{har}, the last equality implies
\begin{equation*}\label{gor}
\dim V_i=\dim W_i,\qquad i=1,2,\ldots
\end{equation*}

If $m:=\max(n-p,n-q)$, then $J_{\circ}^m=J_{\circ}^{\prime m}=0$, and so
\[p=\dim V_m=\dim W_m=q.
\]
Thus, $f_*$ and $g_*$ are affine bijections $V_*\to V_*$ on the same space
\begin{equation*}\label{syu}
V_*:=V_m=W_m=\mathbb F^p.
\end{equation*}
By \eqref{hyr}, the restriction of $h$ to $V_*$ gives some biregular map  $h_*:V_* \to V_*$.
Restricting the equality $hf=g h$ to $V_*$, we obtain
$h_* f_*=g_* h_*.$
Therefore, $f_*$ and $g_*$ are biregularly conjugate.
By Theorem~\ref{Blanc}, their matrices $D_{\alpha}$ and $D_{\beta }$ \emph{coincide up to permutation of eigenvalues}.

The nilpotent Jordan matrices
$J_{\circ}$ and $J_{\circ}'$ \emph{coincide up to permutation of blocks}
since by \eqref{suk}
the number of their Jordan blocks is equal to $n-\dim V_1$, the number of their Jordan blocks of size $\ge2$ is equal to $(n-\dim V_2)-(n-\dim V_1)$, the number of their Jordan blocks of size $\ge3$ is equal to $(n-\dim V_3)-(n-\dim V_2)$, and so on.


\begin{thebibliography}{1}

\bibitem{Blanc}
J. Blanc,  Conjugacy classes of affine automorphisms of $\mathbb K^n$ and linear automorphisms of $\mathbb P^n$ in the Cremona groups, Manuscripta Math. 119 (2006) 225--241.


\bibitem{st-1}
T.V. Budnitska, Classification of topological conjugate affine mappings, Ukrainian Math. J. 61 (2009)  164--170.

\bibitem{har} R. Hartshorne, Algebraic Geometry, Graduate Texts in Mathematics 52, Springer-Verlag, 1997.


\end{thebibliography}
\end{document}